\documentclass{amsart}
\usepackage{graphicx} 
\usepackage{amsmath,dsfont}
\usepackage{amsthm}
\usepackage{amssymb}
\usepackage{xcolor} 
\usepackage{colortbl} 
\usepackage{tikz} 
\usepackage{hf-tikz} 
\usepackage{makecell}
\usepackage{mathrsfs}
\usepackage{graphicx}
\usepackage{enumerate}\usepackage{chngcntr}

\usepackage{tikz}
\usepackage[colorlinks,
linkcolor=red,
anchorcolor=blue,
citecolor=green
]{hyperref}
\newtheorem{theorem}{Theorem}
\newtheorem*{theorem*}{Theorem}
\numberwithin{theorem}{section}
\newtheorem{definition}[theorem]{Definition}
\newtheorem{prop}{Proposition}
\newtheorem{lemma}[theorem]{Lemma}
\newtheorem{remark}[theorem]{Remark}
\newtheorem{conjecture}[theorem]{Conjecture}
\newtheorem{corollary}[theorem]{Corollary}
\numberwithin{prop}{section}
\counterwithout{table}{section}

\title[Curved Kakeya sets and Nikodym problems]{Curved Kakeya sets and Nikodym problems on manifolds}
\author{Chuanwei Gao}
\address{School of Mathematical Sciences, Capital Normal University, Beijing 100048, China}
\email{cwgao@cnu.edu.cn}

\author{Diankun Liu}
\address{School of Mathematical Sciences, Zhejiang University, Hangzhou 310027, PR China}
\email{12335023@zju.edu.cn}
\author{Yakun Xi}
\address{School of Mathematical Sciences, Zhejiang University, Hangzhou 310027, PR China}
\email{yakunxi@zju.edu.cn}

\begin{document}
	\maketitle
	\begin{abstract}
		In this paper, we study curved Kakeya sets associated with phase functions satisfying Bourgain's condition. In particular, we show that the analysis of curved Kakeya sets arising from translation-invariant phase functions under Bourgain's condition, as well as Nikodym sets on manifolds with constant sectional curvature, can be reduced to the study of standard Kakeya sets in Euclidean space. Combined with the recent breakthrough of Wang and Zahl \cite{WangZahl}, our work establishes the Nikodym conjecture for three-dimensional manifolds with constant sectional curvature.
		
		Moreover, we consider $(d,k)$-Nikodym sets and $(s,t)$-Furstenberg sets on Riemannian manifolds. For manifolds with constant sectional curvature, we prove that these problems can similarly be reduced to their Euclidean counterparts. As a result, the Furstenberg conjecture on two-dimensional surfaces with constant Gaussian curvature follows from the work of Ren and Wang \cite{Ren-Wang}.
		
	\end{abstract}
	\section{Introduction}\label{sec intro}
	\subsection{Kakeya sets, Nikodym sets and curved Kakeya sets}
	Let us begin by recalling the notion of standard Kakeya sets in $\mathbb{R}^d$.
	\begin{definition}
		A \textbf{Kakeya set} in $\mathbb{R}^d$ is a compact set $E \subset \mathbb{R}^d$ such that for every direction $e \in S^{d-1}$, there exists a point $x \in \mathbb{R}^d$ for which
		\[
		x + t e \in E 
		\quad
		\text{for all } t \in \left[-\tfrac12, \tfrac12\right].
		\] 
	\end{definition}
	In 1920, Besicovitch constructed an example of a Kakeya set in the plane with Lebesgue measure zero. Since then, mathematicians have extensively studied the dimensions of Kakeya sets, and the following conjecture becomes a central problem in harmonic analysis and geometric measure theory.
	\begin{conjecture}
		The Hausdorff dimension of the Kakeya sets is $d$ in $\mathbb{R}^{d}$. 
	\end{conjecture}
	
	Davies \cite{D} (see also \cite{Cor77}) proved this conjecture in $\mathbb{R}^{2}$ in the 1970s. Bourgain \cite{Bourgain91.2} established that the Hausdorff dimension of these sets is at least $\frac{7}{3}$ in $\mathbb{R}^{3}$, employing a \emph{bush argument}. Schlag \cite{Schlag} subsequently re-proved this result by a different method in 1998. In 1995, Wolff \cite{Wolff95} improved the lower bound for the Hausdorff dimension of Kakeya sets to $\frac{d+2}{2}$ in $\mathbb{R}^{d}$, introducing an innovative technique now known as the \emph{hairbrush argument}. In 2002, Katz, \L aba, and Tao \cite{KLT} obtained a breakthrough by improving the Minkowski dimension of Kakeya sets to $\tfrac{5}{2} + \epsilon_{0}$ in dimension $3$, where $\epsilon_{0} > 0$ is an absolute constant. Subsequently, Katz and Zahl \cite{Katz19} showed the Hausdorff dimension in dimension $3$ also exceeds $\tfrac{5}{2} + \epsilon_0$, for a different absolute constant $\epsilon_{0} \textgreater 0$.  Later, Katz and Zahl \cite{Katz21} pushed the Hausdorff dimension in $\mathbb{R}^4$ up to $3.059$ using a \emph{planebrush argument}. Most recently, in a series of articles \cite{sticky,assouad,WangZahl}, Wang and Zahl made a landmark discovery by proving the Kakeya conjecture in $\mathbb{R}^{3}$ in full, confirming that every Kakeya set in three dimensions has Hausdorff dimension $3$.
	
	For a high-dimensional Kakeya set, in 1999, Bourgain \cite{Bourgain99} showed that the Hausdorff dimension is at least 
	$\frac{13}{25}d + \frac{12}{25}$, 
	and the Minkowski dimension is at least 
	$\frac{6d}{11} + \frac{5}{11}$, 
	using the arithmetic method. Katz and Tao \cite{Katz-Tao} proved that Kakeya sets in $\mathbb{R}^{d}$ have Hausdorff dimension greater than or equal to $(2-\sqrt{2})(d-4)+3$. Hickman, Rogers, and Zhang \cite{HRZ} and Zahl \cite{Zahl} proved a lower bound on the Hausdorff dimension related to the space dimension. As a corollary of their result, for every $\epsilon > 0$, there exists an infinite sequence of dimensions $d$ such that every Kakeya set $K \subset \mathbb{R}^{d}$ satisfies
	\[
	\dim_{\mathcal{H}}K \geq (2 - \sqrt{2})d + \frac{3}{2} - \frac{1}{\sqrt{2}} - \epsilon,
	\]
	where $\dim_{\mathcal{H}}K$ denotes the Hausdorff dimension of the set $K$.
	
	In Euclidean space, Nikodym sets naturally arise as the dual counterparts to Kakeya sets. However, unlike Kakeya sets, the concept of Nikodym sets lends itself to a natural generalization in the Riemannian setting. Minicozzi and Sogge \cite{MS,Sogge99} introduced the Nikodym problem on manifolds and provided a definition of Nikodym sets in this context.

	\begin{definition}
		Let $(M^d,g)$ be a Riemannian manifold of dimension $d$. By possibly rescaling the metric, we may assume that the injectivity radius of $M^d$ is at least 10. Consider a Borel set $\Omega \subset M^d$, $0 < \lambda < 1$. Define
		\begin{multline}
			\Omega_{\lambda}^{\star}
			\,=\,
			\bigl\{
			x \in M^d
			:\,
			\exists\, \gamma_{x} \,
			\text{with }\,
			\lvert \gamma_{x} \cap \Omega \rvert
			\,\geq\,\lambda
			\lvert \gamma_{x}\rvert,\\  \text{ where $\gamma_{x}$ is a unit geodesic segment passing through } x\text{} 
			\bigr\}.
		\end{multline}
		We say that $\Omega$ is a \textbf{Nikodym set} if for all $\lambda$ sufficiently close to $1$, the set $\Omega_{\lambda}^{\star}$ has positive measure.
	\end{definition}
	
	In the Euclidean case, the Nikodym problem is mostly equivalent to the Kakeya problem; see, for instance, \cite{Tao99, Mattila}. On the other hand, the Nikodym sets behave differently for general manifolds. Minicozzi and Sogge \cite{MS} demonstrated that, for general Riemannian manifolds $M^{d}$, the fundamental lower bound of $\lceil\frac{d+1}{2}\rceil$ for Nikodym sets is generally sharp. Consequently, one needs to consider a smaller class of Riemannian manifolds if one wants a Nikodym set to have full dimension.
	Sogge \cite{Sogge99} proved that the Hausdorff dimension of Nikodym sets on manifolds with constant sectional curvature is at least $\frac{5}{2}$, via a variant of Wolff's hairbrush method. Later, the third author \cite{Xi17} extended this result to higher dimensions, establishing a lower bound of $\frac{d+2}{2}$ on manifolds with constant sectional curvature. As a result, it is natural to conjecture the following.
	
	\begin{conjecture}
		The Hausdorff dimension of a Nikodym set is $d$ on a Riemannian manifold $M^{d}$ with constant sectional curvature.
	\end{conjecture}
	
	Dai, Gong, Guo, and Zhang \cite{DGGZ24} showed that Nikodym sets on manifolds $M^{d}$ with constant sectional curvature satisfy Hausdorff dimension lower bounds matching the lower bounds for the Hausdorff dimension of Kakeya sets in $\mathbb{R}^d$ established in \cite{HRZ}.

	In this paper, we establish direct connections between Nikodym sets on manifolds with constant sectional curvature and standard Kakeya sets in $\mathbb{R}^d$, thereby improving the known lower bounds on the Hausdorff dimension of Nikodym sets in certain dimensions.
	
	\begin{theorem}\label{theomain2}
		Suppose that in $\mathbb{R}^d$, any Kakeya set must have Hausdorff (resp. Minkowski) dimension at least $\alpha$. Then any Nikodym set on a manifold $M^d$ with constant sectional curvature must also have Hausdorff (resp. Minkowski) dimension at least $\alpha$. In particular, the Kakeya conjecture in $\mathbb{R}^d$ implies the Nikodym conjecture on manifolds $M^d$ with constant sectional curvature for any dimension $d \ge 2$.
	\end{theorem}

	As a key consequence, we prove the Nikodym conjecture for $3$-dimensional manifolds with constant sectional curvature, leveraging the recent breakthrough of Wang and Zahl \cite{WangZahl}.
	\begin{corollary}
		Every Nikodym set on a Riemannian manifold $M^{3}$ with constant sectional curvature has Hausdorff dimension $3$.
	\end{corollary}
	
	Both the Kakeya sets in Euclidean space and the Nikodym sets on manifolds can be unified under the framework of \emph{curved Kakeya sets}. To this end, we consider the phase function of a H\"{o}rmander oscillatory integral operator:
	$$T_{N}f(\mathbf{x})=\int e^{iN\phi(\mathbf{x};y)}\,a(\mathbf{x} ;y)\,f(y)\,dy,\quad \text{where } (\mathbf{x}; y) = (x, t; y) \in \mathbb{B}_{\varepsilon_{0}}^{d-1} \times \mathbb{B}_{\varepsilon_{0}}^{1} \times \mathbb{B}_{\varepsilon_{0}}^{d-1}.$$
	Here $a(\mathbf{x}; y)$ is supported on $\mathbb{B}_{\varepsilon_{0}}^{d-1} \times \mathbb{B}_{\varepsilon_{0}}^{1} \times \mathbb{B}_{\varepsilon_{0}}^{d-1}$, where $\mathbb{B}_{\varepsilon_{0}}^{d}$ denotes a $d$-dimensional ball of radius $\varepsilon_{0}$ with $\varepsilon_{0}$ being a  small constant that depends on $\phi$. We impose the following non-degeneracy condition on $\phi$.
	\begin{definition}
		[H\"{o}rmander's non-degeneracy condition] We say $\phi$ is a  non-degenerate phase function, if it satisfies the following hypothesis.
		\\$(H_{1})$ ${\rm rank\,} \nabla_{\mathbf{x}}\nabla_{y}\phi(\mathbf{x} ;y)=d-1$ for all $ (\mathbf{x},y)=(x,t;y) \in  \mathbb{B}_{\varepsilon_{0}}^{d-1} \times \mathbb{B}_{\varepsilon_{0}}^{1} \times \mathbb{B}_{\varepsilon_{0}}^{d-1};$
		\\$(H_{2})$  If 
		$$ G_{0}(\mathbf{x};y):=\bigwedge_{j=1}^{d-1}\partial_{y_{j}}\nabla_{\mathbf{x}}\phi(\mathbf{x};y),$$
		then 
		$$\det\nabla_{y}^{2}\langle \nabla_{\mathbf{x}}\phi(\mathbf{x};y),G_{0}(\mathbf{x};y_{0})\rangle|_{y=y_{0}}\neq 0$$
		holds in the support of $a$.
	\end{definition}
	By a standard argument in \cite{Hor73}, a phase satisfying $(H_{1})$ and $(H_{2})$ can be reduced to the normal form
	\[
	\phi(\mathbf{x}; y) = \langle x,y\rangle + t \langle y, A y \rangle + O(|\mathbf{x}|^{2} |y|^{2} + |t| |y|^{3}),
	\]
	where $(\mathbf{x}; y) = (x, t; y) \in \mathbb{B}_{\varepsilon_{0}}^{d-1} \times \mathbb{B}_{\varepsilon_{0}}^{1} \times \mathbb{B}_{\varepsilon_{0}}^{d-1}$, and $A$ is a non-degenerate matrix.
	
	Every non-degenerate phase $\phi$ determines a family of Kakeya curves \cite{Wisewell05}.
	
	\begin{definition}[Kakeya Curves]\label{curved}
		For $y \in \mathbb{B}_{\varepsilon_{0}}^{d-1}$, $\mathbf{x} \in \mathbb{B}_{\varepsilon_{0}}^{d-1} \times \mathbb{B}_{\varepsilon_{0}}^{1}$, and $0 < \delta < \varepsilon_{0}$, define the Kakeya curves associated with $\phi$ as
		$$
		\Gamma_{y}^{\phi}(\mathbf{x}) = \{ \mathbf{x}' \in \mathbb{B}_{\varepsilon_{0}}^{d-1} \times \mathbb{B}_{\varepsilon_{0}}^{1} : \nabla_{y}\phi(\mathbf{x}';y) = \nabla_{y}\phi(\mathbf{x};y) \},
		$$
		and the associated $\delta$-tubes
		$$
		T_{y}^{\delta, \phi}(\mathbf{x}) = \{ \mathbf{x}' \in \mathbb{B}_{\varepsilon_{0}}^{d-1} \times \mathbb{B}_{\varepsilon_{0}}^{1} : |\nabla_{y}\phi(\mathbf{x}';y) - \nabla_{y}\phi(\mathbf{x};y)| < \delta \}.
		$$
	\end{definition}
	
	\begin{definition}[Curved Kakeya sets]\label{curvedK}
		A set $E \subset \mathbb{R}^d$ is called a \emph{curved Kakeya set} (associated with $\phi$) if for every $y \in \mathbb{B}_{\varepsilon_0}^{d-1}$, there exists an $\omega \in \mathbb{B}_{\varepsilon_0}^{d-1}$ such that 
		$$
		\Gamma^\phi_{y}((\omega, 0)) \subset E.
		$$
	\end{definition}

	Clearly, the notion of a curved Kakeya set generalizes the notion of a standard Kakeya set in $\mathbb{R}^d$. Indeed, if we take the phase function as  
	$$
	\phi(x, t; y) = \langle x,y\rangle + t\,h(y),
	$$ 
	where $h(y)$ has a non-degenerate Hessian, the corresponding curves are straight lines in $\mathbb{R}^d$, as in the standard Kakeya problem. It is known (see, e.g., \cite{MS, DGGZ24}) that a Nikodym set on manifolds forms a curved Kakeya set when the phase function is chosen to be the Riemannian distance function.

	Wisewell \cite{Wisewell05} showed that the sharp Hausdorff dimension lower bound for general curved Kakeya sets is $\frac{d+1}{2}$ in $\mathbb{R}^d$ for odd dimensions $d$. Additionally, Bourgain and Guth \cite{BG11} improved this exponent to $\frac{d+2}{2}$ for even dimensions $d$.
	
	These bounds on Hausdorff dimension are sharp in the sense that there exist some curved Kakeya sets whose Hausdorff dimension is exactly $\lceil\frac{d+1}{2}\rceil$. The associated curved Kakeya set may concentrate on a curved $\lceil\frac{d+1}{2}\rceil$-dimensional surface, a phenomenon now known as Kakeya compression. More shockingly, even if we consider translation-invariant phase functions, the associated curved Kakeya set could have Hausdorff dimension $\lceil \frac{d+1}{2} \rceil$. To be more specific, we introduce the translation-invariant condition.
	
	\begin{definition}[Translation-invariant condition \cite{Bourgain91,CGGHIW24}]
		A non-degenerate phase $\phi:\mathbb{B}_{\varepsilon_{0}}^{d-1} \times \mathbb{B}_{\varepsilon_{0}}^{1}\times\mathbb{B}_{\varepsilon_{0}}^{d} \rightarrow \mathbb{R}$ is \emph{translation-invariant} if it is of the form
		$$
		\phi(x,t;y)=\langle x,y\rangle+\psi(t;y)
		$$
		for some $\psi \in C^{\infty}((-\varepsilon_0,\varepsilon_0)\times \mathbb{B}_{\varepsilon_{0}}^{d-1} )$ satisfying $\psi(0;y)=0$ for all $y \in \mathbb{B}_{\varepsilon_{0}}^{d-1}$.
	\end{definition}
	
	Bourgain \cite{Bourgain91} constructed a translation-invariant phase function in $\mathbb{R}^3$ with extreme Kakeya compression, given by
	$$
	\phi(x,t;y) = x_1 y_1 + x_2 y_2 + t\,y_1 y_2 + \frac{t^2}{2}\,y_1^2.
	$$  
	Many more similar examples have since been constructed. To surpass these generally sharp bounds, additional conditions on the phase function are required.

	Guo, Wang, and Zhang \cite{GWZ22} introduced a condition on phase functions, known as Bourgain's condition. Under this condition, they obtained improved associated restriction estimates and conjectured that the critical exponent coincides with that of the restriction conjecture when the phase function satisfies Bourgain's condition. In particular, they demonstrated that the polynomial Wolff axiom holds under Bourgain's condition, providing an effective tool to avoid Kakeya compression.

	\begin{definition}
		[Bourgain's condition \cite{Bourgain91,GWZ22}]
		Let $\phi$ be a phase function satisfying H\"{o}rmander's non-degeneracy condition. We say that it satisfies Bourgain's condition at $(\mathbf{x}_{0};y_{0})$ if
		\begin{equation}
			(G_{0} \cdot \nabla_{\mathbf{x}})^{2}\partial_{y_{i}y_{j}}^{2}\phi(\mathbf{x}_{0};y_{0})= C(\mathbf{x}_{0};y_{0})(G_{0}\cdot \nabla_{\mathbf{x}})\partial_{y_{i}y_{j}}^{2}\phi(\mathbf{x}_{0};y_{0})
		\end{equation}
		for any $1 \leq i,j \leq d-1$. Here, $C(\mathbf{x}_0;y_0)$ is a function which depends on $\mathbf{x}_0$ and $y_0$.
	\end{definition}
	
	Dai, Gong, Guo, and Zhang \cite{DGGZ24} studied curved Kakeya sets for phase functions satisfying Bourgain's condition. They proved that the Riemannian distance function on a Riemannian manifold with a real analytic metric satisfies Bourgain's condition if and only if the sectional curvature is constant. This represents a remarkable connection between the Nikodym problem on manifolds with constant curvature and the curved Kakeya problem under Bourgain's condition. They also showed that the curved Kakeya sets associated with Bourgain's condition have the same Hausdorff dimension as the standard Kakeya sets in \cite{HRZ}. 
	
	It is worth noting that, recently, Chen, Gan, Guo, Hickman, Iliopoulou, and Wright \cite{CGGHIW24} studied real analytic phase functions that satisfy the translation-invariant condition and a certain Kakeya non-compression condition. They proved that the Hausdorff dimension of the corresponding curved Kakeya sets is strictly greater than $2$ in $\mathbb{R}^{3}$. Their Kakeya non-compression condition is a general condition which is rather weak, as it excludes the case where a Kakeya set concentrates on a hypersurface. In addition, the Kakeya non-compression condition is stable with respect to perturbations, while Bourgain's condition is in general unstable with respect to perturbations (see, e.g., \cite{MS,Sogge-Xi}).

	We consider translation-invariant phase functions that satisfy Bourgain's condition. We show that the corresponding curved Kakeya set is a standard Kakeya set in disguise via a diffeomorphism. As a result, the curved Kakeya set corresponding to such a phase function inherits all the known results for standard Kakeya sets. 
	
	\begin{theorem}\label{theomain}
		For a translation-invariant phase $\phi$ satisfying Bourgain's condition, there exists a diffeomorphism $\kappa(x,t)$ such that
		$$
		\phi(\kappa(x,t); y) = \langle x,y\rangle + t\,h(y) + q(y)+f(t),
		$$
		where $h(y)$ has a non-degenerate Hessian. Consequently, the corresponding curved Kakeya sets are mapped via a diffeomorphism to standard Kakeya sets in $\mathbb{R}^d$, thereby inheriting all the known results for standard Kakeya sets.
	\end{theorem}

	Combining the recent results of Wang and Zahl \cite{WangZahl} with the theorem above, we obtain the following corollary:
	
	\begin{corollary} In $\mathbb{R}^{3}$, if a translation-invariant phase function satisfies Bourgain's condition, then the corresponding curved Kakeya sets have Hausdorff dimension $3$. \end{corollary}
	
	To sum it up, in Theorem \ref{theomain2} and \ref{theomain}, we reduce two types of curved Kakeya problems under Bourgain's condition to the standard Kakeya problem in $\mathbb{R}^{d}$ for any dimension $d \geq 3$. Specifically, in both cases, the corresponding curved Kakeya problem inherits all the known dimension bounds associated with the standard Kakeya problem, as well as the associated maximal function bounds. See Proposition \ref{Prop} and \ref{Prop4}.

	\subsection{Other Nikodym-type sets on manifolds}
	
	Falconer \cite{Falconer} demonstrated the existence of a $(d,d-1)$-Nikodym set in $\mathbb{R}^{d}$. Specifically, there exists a Borel set $N \subset \mathbb{R}^{d}$ with Lebesgue measure zero such that for every $x \in \mathbb{R}^{d} \setminus N$, there is a hyperplane $V$ passing through $x$ with $V \setminus \{x\} \subset N$. Mitsis \cite{Mitsis} further established that these sets have full Hausdorff dimension.
	
	The notion of a $(d,k)$-Nikodym set naturally extends to a Riemannian manifold locally via the exponential map (see the definitions in Section \ref{Nikodym type}). Moreover, we prove the existence of measure zero $(d,d-1)$-Nikodym sets on manifolds with constant sectional curvature and show that their Hausdorff dimension equals $d$.
	
	\begin{theorem}
		For a $d$-dimensional manifold with constant sectional curvature, $(d,d-1)$-Nikodym sets exist. Moreover, every $(d,d-1)$-Nikodym set has Hausdorff dimension $d$.
	\end{theorem}
	
	Furstenberg introduced an interesting problem in the plane that generalizes both the Kakeya and the Nikodym problems. He considered a set $E$ with the following property: for $s \in (0,1]$ there exists a family of lines $\mathcal{L}$ in the plane, covering every direction, such that 
	$
	\dim_{\mathcal H}(E \cap l) \geq s
	$
	for every $l \in \mathcal{L}$. He conjectured that any such set $E$ must have Hausdorff dimension at least 
	$
	\frac{3s+1}{2}.
	$
	Later, these sets were generalized to $(s,t)$-Furstenberg sets.
	
	\begin{definition}[$(s,t)$-Furstenberg Sets]
		For $s \in (0,1]$ and $t \in (0,2]$, an $(s,t)$-Furstenberg set is a set $E \subset \mathbb{R}^{2}$ with the following property: there exists a family $\mathcal{L}$ of lines with 
		$
		\dim_{\mathcal H}\mathcal{L} \geq t
		$
		such that 
		$
		\dim_{\mathcal H}(E \cap l) \geq s
		$
		for every $l \in \mathcal{L}$, where $\dim_{\mathcal H}\mathcal{L}$ is defined as the Hausdorff dimension of the family of lines viewed as a subset of the affine $1$-subspaces of $\mathbb{R}^{2}$.
	\end{definition}
	
	It is now a theorem of Ren and Wang \cite{Ren-Wang} that any $(s,t)$-Furstenberg set $E$ satisfies
	$$
	\dim_{\mathcal H} E \geq \min\Bigl\{s + t,\; \frac{3s + t}{2},\; s + 1\Bigr\}.
	$$
	
	In Section \ref{Nikodym type}, we extend the notion of $(s,t)$-Furstenberg sets to Riemannian surfaces and establish the Furstenberg estimate for surfaces with constant Gaussian curvature.
	
	\begin{theorem}\label{theoFur}
		A $(s,t)$-Furstenberg set $E$ on a surface with constant Gaussian curvature satisfies
		$$
		\dim_{\mathcal H} E \geq \min\Bigl\{s + t,\; \frac{3s + t}{2},\; s + 1\Bigr\}.
		$$
	\end{theorem}
	In a recent work of Wang and Wu \cite{WangWu2024}, a new framework was developed to study the Fourier restriction conjecture using refined decoupling estimates alongside Furstenberg estimates. Noting that relevant variable coefficient refined decoupling estimates have been established in \cite{IosevichLiuXi2022}, it would be interesting to explore whether similar ideas could be employed in conjunction with Theorem \ref{theoFur} to address restriction-type questions on manifolds with constant sectional curvature.
	
	\subsection{Organization} This paper is organized as follows. In Section \ref{sec Nik}, we investigate the Nikodym problem on manifolds with constant sectional curvature and establish its connection to the standard Kakeya problem by showing that, locally, the geodesics can be straightened via a diffeomorphism. Section \ref{Nikodym type} is devoted to the study of other Nikodym-type sets on manifolds, where we introduce and analyze $(d,k)$-Nikodym sets as well as $(s,t)$-Furstenberg sets, deriving sharp Hausdorff dimension estimates on space forms. In Section \ref{sec trans}, we focus on translation-invariant phase functions that satisfy Bourgain's condition, proving that the associated curved Kakeya sets can be mapped, through an appropriate diffeomorphism, to standard Kakeya sets, thereby inheriting all their known dimensional lower bounds.
	
	\subsection{Acknowledgements.} This project was supported by the National Key R\&D Program of China under Grant No. 2022YFA1007200 and 2024YFA1015400. C. Gao was supported by NSF of China under Grant No. 12301121.  Y. Xi was partially supported by NSF of China under Grant No. 12171424 and Zhejiang Provincial Natural Science Foundation of China under Grant No. LR25A010001. The authors would like to thank Mingfeng Chen, Shengwen Gan, Shaomin Guo, Hong Wang, and Ruixiang Zhang for some interesting conversations on this topic.

	\section{The Nikodym problem on manifolds}\label{sec Nik}
	We establish a connection between Nikodym sets on manifolds with constant curvature and standard Kakeya sets.

	\subsection{Nikodym sets on manifolds}
	Let us consider Nikodym sets on manifolds $M^d$ with constant sectional curvature for any dimension $d \ge 2$. In a small local open neighborhood of a point on $M^d$, we shall find a diffeomorphism that maps all the geodesic segments to straight line segments. This diffeomorphism can be viewed as a way of “straightening” the Nikodym sets.
	
	There are exactly three types of \emph{local} behaviors for manifolds with constant sectional curvature: the sphere $\mathbb{S}^d$, hyperbolic space $\mathbb{H}^d$, and Euclidean space $\mathbb{R}^d$. More precisely, for any such manifold with an injectivity radius of at least $10$, every unit geodesic ball is isometric to a small geodesic ball in $\mathbb{R}^d$, $\mathbb{S}^d$, or $\mathbb{H}^d$, depending on the sign of the curvature. Consequently, it suffices to study the Nikodym sets in an open neighborhood of a point in either $\mathbb{S}^d$ or $\mathbb{H}^d$. In each case, one can choose a local coordinate chart in which all geodesics are represented as straight lines.

	\begin{theorem}\label{theogeodesic}
		Let $B$ be a unit geodesic ball in a Riemannian manifold $M^d$ with constant sectional curvature and injectivity radius at least 10. Then there exists a diffeomorphism mapping $B$ onto a ball in $\mathbb{R}^d$ such that all geodesics are mapped to straight lines.
	\end{theorem}
	
	\begin{remark}
		A classical result in Riemannian geometry, originally due to Beltrami, shows that the ability to locally map geodesics to straight lines via a diffeomorphism characterizes constant sectional curvature. We provide an explicit constructive proof to Theorem \ref{theogeodesic} below because this construction is central to several of our main results. Moreover, Dai, Gong, Guo, and Zhang \cite{DGGZ24} demonstrated that the Riemannian distance function on a real analytic manifold satisfies Bourgain's condition if and only if the sectional curvature is constant. In other words, for real anlytic metrics, Bourgain's condition holds for the Riemannian distance function precisely when all geodesics can be locally straightened by a diffeomorphism.
	\end{remark}

	\begin{proof}
		As discussed above, it suffices to prove this theorem for $M^d = \mathbb{S}^d$ and $M^d = \mathbb{H}^d$. First, we consider the sphere
		$$
		\mathbb{S}^d = \{ (x_1, \dots, x_d, x_{d+1}) : x_1^2 + \cdots + x_d^2 + x_{d+1}^2 = 1 \},
		$$
		and restrict our attention to the neighborhood of the point
		$
		p_0 = (0, \dots, 0, -1).
		$

		We construct the map $P$ from a neighborhood of $p_0$ on the sphere to the plane $\{x_{d+1} = -1\}$. Given a point $p \in \mathbb{S}^d$, 
		$$
		p = (x_1, \dots, x_d, x_{d+1}),
		$$
		we define the diffeomorphism $P(p)$ by
		$$
		P(p) = (u_1, \dots, u_d, -1) = \Bigl( -\frac{x_1}{x_{d+1}}, \dots, -\frac{x_d}{x_{d+1}}, -1 \Bigr).
		$$
		Since we may require $-1 < x_{d+1} < 0$ in the chosen neighborhood, this is well-defined. To show that $P$ is a diffeomorphism, we consider its inverse map $F : \mathbb{R}^{d} \to \mathbb{S}^d \subset \mathbb{R}^{d+1}$ defined by
		$$
		F(u_{1}, \dots, u_{d}) = \Bigl(\frac{u_{1}}{\sqrt{u_{1}^{2}+\cdots+u_{d}^{2}+1}}, \dots, \frac{u_{d}}{\sqrt{u_{1}^{2}+\cdots+u_{d}^{2}+1}}, \frac{-1}{\sqrt{u_{1}^{2}+\cdots+u_{d}^{2}+1}}\Bigr).
		$$
		It is standard to check that any \( d \) vectors chosen from the gradient vectors \[\{\nabla F^1, \dots, \nabla F^{d+1}\}\] are linearly independent, and therefore \( P \) is indeed a local diffeomorphism.
		
		Next, consider any $2$-plane that intersects the hyperplane $X_{d+1} = -1$ and passes through the origin $\mathbf{O}$. Let $$\{(a_{1,1}, \dots, a_{1,d+1}), (a_{2,1}, \dots, a_{2,d+1}), \dots, (a_{d-1,1}, \dots, a_{d-1,d+1})\}$$ be $d-1$ linearly independent normal vectors to this $2$-plane; note that these vectors are linearly independent of $(0,\dots,0,1)$. For a geodesic on the sphere, with coordinates $(x_1,\dots,x_{d+1})$, we have the system of equations
		\begin{equation}
			\begin{pmatrix}
				a_{1,1} & a_{1,2} & \cdots & a_{1,d+1} \\
				\vdots & \vdots & \ddots & \vdots \\
				a_{d-1,1} & a_{d-1,2} & \cdots & a_{d-1,d+1} \\
			\end{pmatrix}
			\begin{pmatrix}
				x_{1} \\ x_{2} \\ \vdots \\ x_{d+1}
			\end{pmatrix}
			=
			\begin{pmatrix}
				0 \\ 0 \\ \vdots \\ 0
			\end{pmatrix}.
		\end{equation}
		By substituting the relation $u_i = -\frac{x_i}{x_{d+1}}$ for $i=1,\dots,d$, we obtain
		\begin{equation}\label{matri}
			\begin{pmatrix}
				a_{1,1} & a_{1,2} & \cdots & a_{1,d} \\
				\vdots & \vdots & \ddots & \vdots \\
				a_{d-1,1} & a_{d-1,2} & \cdots & a_{d-1,d} \\
			\end{pmatrix}
			\begin{pmatrix}
				u_{1} \\ u_{2} \\ \vdots \\ u_{d}
			\end{pmatrix}
			=
			\begin{pmatrix}
				a_{1,d+1} \\ a_{2,d+1} \\ \vdots \\ a_{d-1,d+1}
			\end{pmatrix}.
		\end{equation}
		It is clear from \eqref{matri} that the trajectory of $(u_1, u_2, \dots, u_d)$ is a straight line in $\mathbb{R}^{d}$. In other words, for any geodesic on the sphere, there exists a $2$-plane through the origin such that this plane intersects the sphere along the geodesic and the hyperplane $X_{d+1}=-1$ in a straight line (see Figure \ref{sphere}).

		\begin{figure}[h] 
			\centering
			\includegraphics[width=0.8\textwidth]{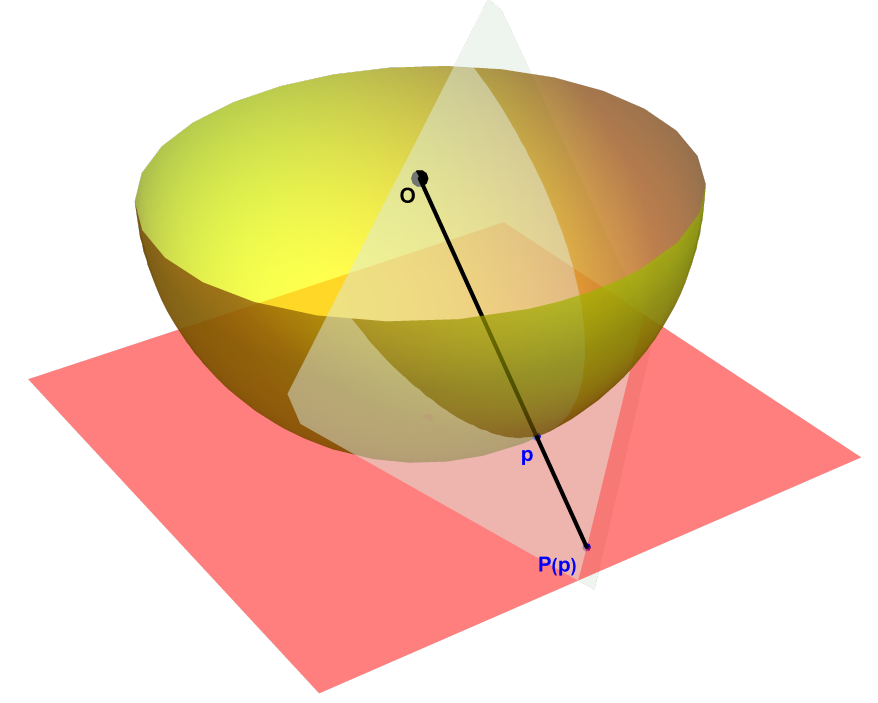}  
			\caption{The straightening map on the sphere}
			\label{sphere} 
		\end{figure}

		Now we consider the hyperbolic space $\mathbb{H}^d$. We shall make use of the Beltrami–Klein model (see, e.g., \cite[page 62 and page 138]{John}). 
		
		\begin{prop}[Beltrami–Klein model]
			The model $\mathbb{K}^{d}$ of $\mathbb H^d$ is a ball of radius $1$ centered at the origin in $\mathbb{R}^{d}$, equipped with the metric given in coordinates $(w_{1},\dots,w_{d})$ by
			$$
			\tilde g=\frac{(dw_{1})^{2}+\cdots+(dw_{d})^{2}}{1-|w|^{2}}+\frac{(w_{1}dw_{1}+\cdots+w_{d}dw_{d})^{2}}{(1-|w|^{2})^{2}}.
			$$
			In this model, any geodesic is represented as a line segment whose endpoints both lie on $\partial \mathbb{K}^{d}$.
		\end{prop}
		
		The Beltrami–Klein model can be similarly realized as a map from a hyperboloid in $\mathbb{R}^{d,1}$ to $\mathbb{R}^{d}$, where $\mathbb{R}^{d,1}$ is $\mathbb{R}^{d+1}$ equipped with the Minkowski metric. We now consider the hyperboloid model of $\mathbb H^d$ given by
		$$
		\mathbb{H}^d = \Bigl\{ (x_1, \dots, x_d, x_{d+1}) : x_1^2+\cdots+x_d^2-x_{d+1}^2=-1;\; x_{d+1}>0 \Bigr\} \subset \mathbb{R}^{d,1}.
		$$
		We define the diffeomorphism 
		$$
		T:\mathbb{H}^{d}\to \mathbb{K}^{d}
		$$ 
		by
		$$
		T(p) = \Bigl(\frac{x_{1}}{x_{d+1}},\dots,\frac{x_{d}}{x_{d+1}}\Bigr),
		$$
		where $p=(x_{1},\dots,x_{d},x_{d+1})\in\mathbb{H}^{d}$.
		
		To show that $T$ is a diffeomorphism, we consider its inverse map 
		$$
		K: \mathbb{R}^{d}\to \mathbb{H}^{d}\subset\mathbb{R}^{d,1}
		$$
		defined by
		$$
		K(w_{1},\dots,w_{d}) = \Bigl(\frac{w_{1}}{\sqrt{1-|w|^2}},\dots,\frac{w_{d}}{\sqrt{1-|w|^2}},\frac{1}{\sqrt{1-|w|^2}}\Bigr),
		$$
		where $|w|^2=w_{1}^{2}+\cdots+w_{d}^{2}$. It is standard to check that any \( d \) vectors chosen from the gradient vectors \(\{\nabla K^1, \dots, \nabla K^{d+1}\}\) are linearly independent, and therefore \( T \) is indeed a local diffeomorphism.
		
		This diffeomorphism maps $\mathbb{H}^{d}$ in $\mathbb{R}^{d,1}$ to the unit ball $\mathbb{K}^{d}$ in the hyperplane 
		$$
		\{(x_{1},\dots,x_{d},x_{d+1}) : x_{d+1}=1\}.
		$$
		For any $2$-plane that intersects this hyperplane and passes through the origin $\mathbf{O}$, the intersection of this hyperplane with $\mathbb{H}^{d}$ is a geodesic. Moreover, the intersection of this hyperplane with the hyperplane $\{x_{d+1}=1\}$ is a line segment, as illustrated in Figure \ref{Hyper}. The diffeomorphism $T$ maps this geodesic to a line segment. In other words, when considering hyperbolic space in the Beltrami–Klein model, the geodesic becomes a line segment with endpoints on the boundary $\partial\mathbb{K}^{d}$. See Figure \ref{Klein}.

		\begin{figure}[h] 
			\centering
			\includegraphics[width=0.7\textwidth]{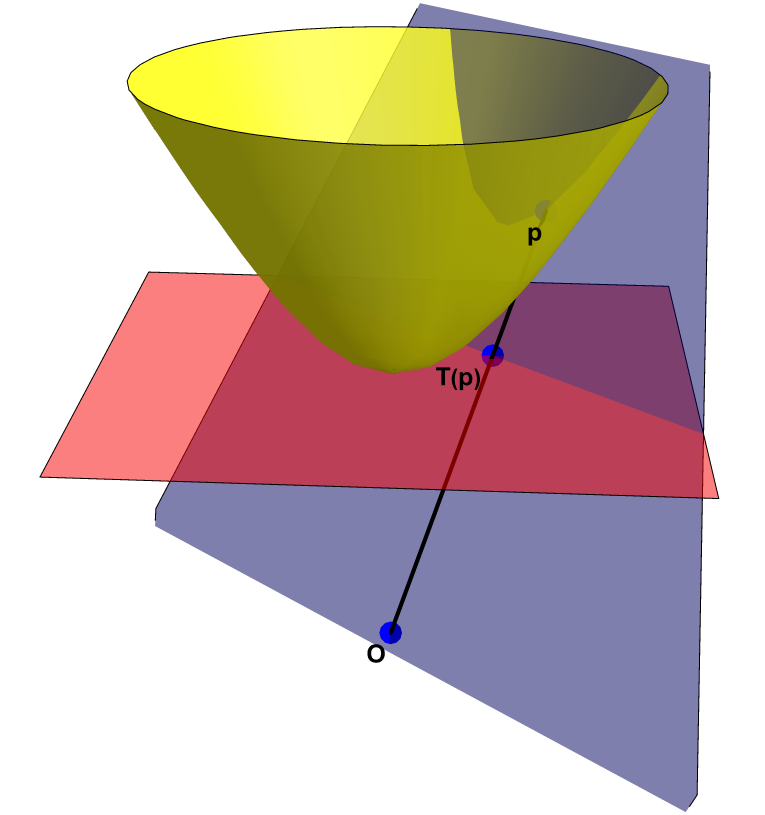}  
			\caption{The straightening map on the hyperbolic space}
			\label{Hyper} 
			
		\end{figure}
		
		\begin{figure}[h] 
			\centering
			\includegraphics[width=0.7\textwidth]{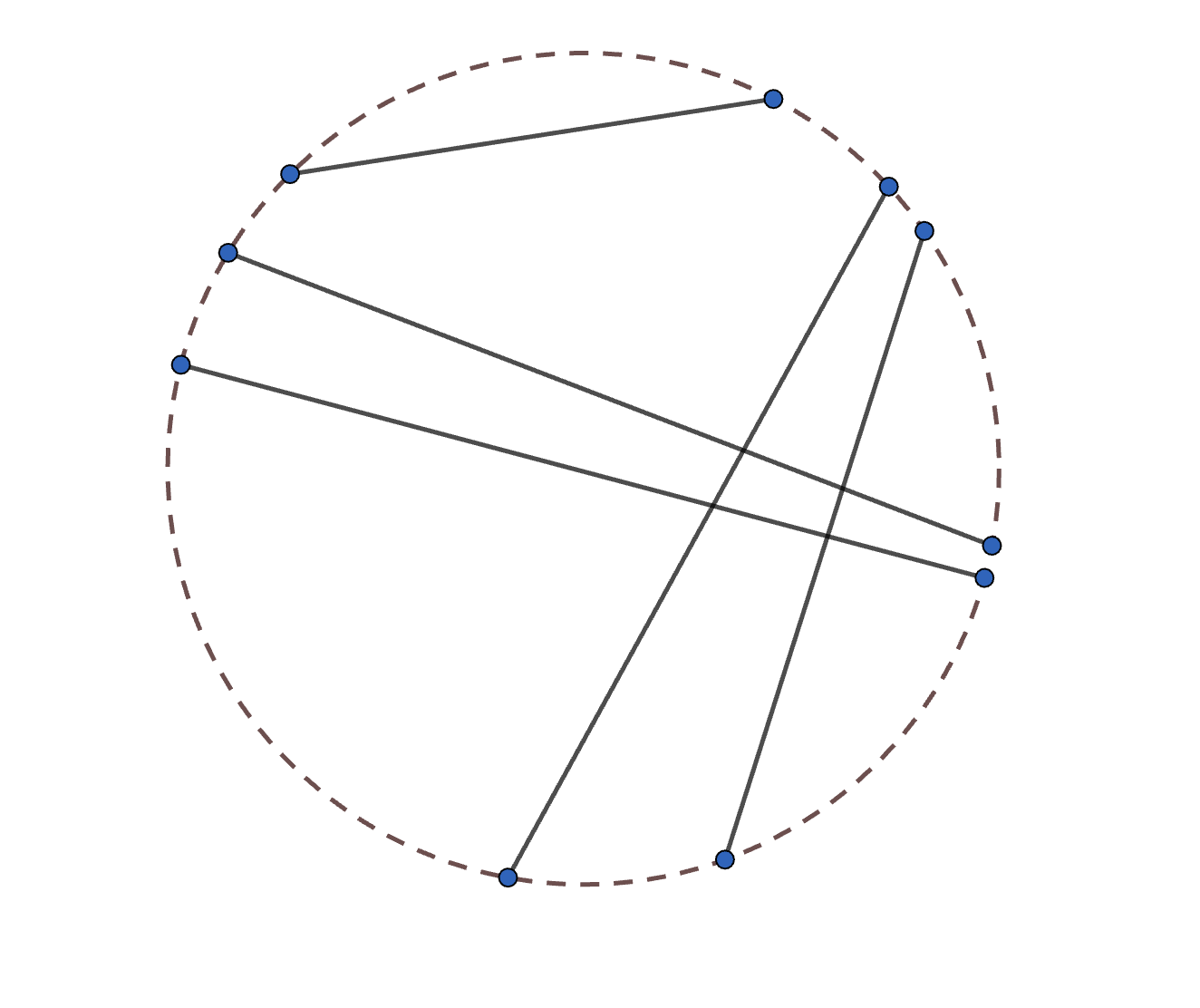}  
			\caption{Geodesics in the Beltrami-Klein model}
			\label{Klein}
		\end{figure}
		
	\end{proof}
	
	By Theorem \ref{theogeodesic}, we can map a Nikodym set on a manifold with constant sectional curvature to a Nikodym set consisting of straight lines, even though the metric on the image is not the Euclidean one. However, this does not affect the Hausdorff dimension, since such a local diffeomorphism is bi-Lipschitz with uniformly bounded Lipschitz constants in any compact set. We summarize this observation in the following proposition.
	
	\begin{prop}\label{propH}
		Let $g$ and $\tilde g$ be two Riemannian metrics on a open set $A \subset \mathbb{R}^d$ with compact closure, and let $N \subset A$ be a subset. Then the Hausdorff (Minkowski) dimension of $N$ with respect to $g$ is equal to that of $N$ with respect to $\tilde g$.
	\end{prop}
	
	\begin{proof}We only treat the Hausdorff case; the Minkowski case is simpler and follows from the same idea.
		Fix $\alpha>0$. For $\epsilon>0$, define
		$$
		H_{\alpha}^{\epsilon}(N)=\inf\Bigl\{\sum_{j=1}^{\infty}r_{j}^{\alpha}: N\subset \bigcup_{j}B(x_{j},r_{j}),\; r_{j}<\epsilon \Bigr\},
		$$
		where $B(x_j,r_j)$ denotes a geodesic ball of radius $r_j$ centered at $x_j$ with respect to the metric $g$. Then the $\alpha$-dimensional Hausdorff measure of $N$ is
		$$
		H_{\alpha}(N)=\lim_{\epsilon\to 0}H_{\alpha}^{\epsilon}(N).
		$$
		There exists a critical exponent $\alpha_0$ such that $H_{\alpha}(N)=\infty$ for any $\alpha<\alpha_0$ and $H_{\alpha}(N)=0$ for any $\alpha>\alpha_0$. $\alpha_0$ is the Hausdorff dimension of $N$.
		
		Since $A$ has compact closure, the identity map $(A,g) \to (A,\tilde g)$ is bi-Lipschitz. Hence, there exist uniform constants $C_1, C_2 > 0$ such that for every $x \in A$ and every $r>0$, given any geodesic ball $B(x,r)$ in $(A,g)$,
		the corresponding geodesic balls $\tilde B(x,C_1^{-1}r)$ and $\tilde B(x,C_2r)$ in $(A,\tilde g)$ satisfy
		$$
		\tilde B(x,C_1^{-1}r)\subset B(x,r)\subset \tilde B(x,C_2r) \quad \text{for all } x\in A.
		$$
		Thus, any covering $\{B(x_j,r_j)\}$ of $N$ with respect to $g$ gives rise to a covering $\{\tilde B(x_j,C_2r_j)\}$ with respect to $\tilde g$, and any covering $\{\tilde B(x_j,s_j)\}$ of $N$ with respect to $\tilde g$ gives rise to a covering $\{B(x_j,C_1s_j)\}$ with respect to $g$. Therefore the quantity $\sum_{j=1}^{\infty}r_{j}^{\alpha}$ defining $H_{\alpha}^{\epsilon}$ is always comparable for either metric.
		Taking the infimum over all such coverings and then letting $\epsilon\to 0$, it follows that the Hausdorff measures computed with respect to $g$ and $\tilde g$ vanish or are infinite simultaneously. Therefore, the Hausdorff dimension of $N$ is the same under both metrics.
	\end{proof}

	As a consequence of Theorem \ref{theogeodesic} and Proposition \ref{propH}, we can reduce the Nikodym problem on manifolds with constant sectional curvature to the Nikodym problem in $\mathbb{R}^d$. Moreover, using Theorem 11.11 in \cite{Mattila}, one may further reduce the study of a Nikodym set in $\mathbb{R}^d$ to that of a standard Kakeya set.
	
	\begin{theorem}[\cite{Mattila}]
		If $1 \leq s \leq d$, and there exists a Nikodym set in $\mathbb{R}^{d}$ of Hausdorff dimension $s$, then there is a Kakeya set in $\mathbb{R}^{d}$ of Hausdorff dimension $s$. In particular, the Kakeya conjecture implies the Nikodym conjecture.
	\end{theorem}
	
	\begin{proof}
		We include the proof for completeness. There is a projective transformation that maps every Nikodym set to a Kakeya set. Define 
		$$
		F(\tilde x, x_{d})=\frac{1}{x_{d}}(\tilde x,1)
		$$
		for $(\tilde x, x_d) \in \mathbb{R}^{d}$ with $x_d \neq 0$. It is not hard to verify that this projective transformation maps the set of passing points (points through which the lines of the Nikodym set pass) to the set of directions. Consequently, any Nikodym set gives rise to a corresponding Kakeya set in $\mathbb{R}^{d}$.
	\end{proof}
	
	\begin{table}[h] 
		\centering 
		\caption{The best known Hausdorff dimension lower bounds for Nikodym sets on a manifold $M^d$ with constant sectional curvature (a space form), with new results highlighted.}\label{table1} 
		\label{tab:example}
		\begin{tabular}{|c|c|c|c|} 
			\hline 
			$d$ & $\dim_{\mathcal H}\geq$ & Kakeya sets in $\mathbb{R}^{d}$ & Nikodym sets in space forms \\ 
			\hline 
			$3$ & $3$ & \cite{WangZahl} & \cellcolor{yellow}Theorem 1.5 \\ 
			$4$ & $3.059\cdots$ & \cite{Katz21} & \cellcolor{yellow}Theorem 1.5\\ 
			$5$ & $\frac{18}{5}$ & \cite{HRZ, Zahl} & \cite{DGGZ24} \\
			$6$ & $7-2\sqrt{2}$ & \cite{Katz-Tao} & \cellcolor{yellow}Theorem 1.5 \\ 
			$7$ & $\frac{34}{7}$ & \cite{HRZ, Zahl} & \cite{DGGZ24} \\ 
			$8$ & $11-4\sqrt{2}$ & \cite{Katz-Tao} & \cellcolor{yellow}Theorem 1.5 \\ 
			$9$ & $6$ & \cite{HRZ, Zahl} & \cite{DGGZ24} \\ 
			$10$ & $15-6\sqrt{2}$ & \cite{Katz-Tao} & \cellcolor{yellow}Theorem 1.5 \\ 
			$11$ & $17-7\sqrt{2}$ & \cite{Katz-Tao} & \cellcolor{yellow}Theorem 1.5 \\ 
			$12$ & $\frac{31}{4}$ & \cite{HRZ, Zahl} & \cite{DGGZ24} \\ 
			$13$ & $21-9\sqrt{2}$ & \cite{Katz-Tao} & \cellcolor{yellow}Theorem 1.5 \\ 
			$14$ & $9$ & \cite{HRZ, Zahl} & \cite{DGGZ24} \\ 
			$15$ & $25-11\sqrt{2}$ & \cite{Katz-Tao} & \cellcolor{yellow}Theorem 1.5 \\ 
			$\vdots$ & $\vdots$ & $\vdots$ & $\vdots$ \\ 
			\hline 
		\end{tabular}
	\end{table}
	
	In summary, we reduce the Nikodym problem on manifolds with constant sectional curvature to the Kakeya problem. Kakeya sets in $\mathbb{R}^d$ and Nikodym sets in space forms share the same lower bounds for both Hausdorff and Minkowski dimensions. Therefore, we have finished the proof of Theorem \ref{theomain2}. As a consequence, the results for the Nikodym problem on manifolds with constant sectional curvature can be improved, as shown in Table \ref{table1}. In particular, in conjunction with the recent progress by Wang and Zahl \cite{WangZahl}, we now know that every Nikodym set on three-dimensional manifolds with constant curvature has Hausdorff dimension $3$.

	\subsection{The Nikodym maximal function on manifolds}
	
	Using Theorem \ref{theogeodesic}, similar connections for the maximal function versions of the two problems can also be drawn. We first recall the definitions of the Kakeya and Nikodym maximal functions in $\mathbb{R}^{d}$ (see, e.g., \cite{Bourgain91.2}).
	
	The \emph{Kakeya maximal function} is defined by
	$$
	K_{\delta}f(\omega)=\sup_{x \in \mathbb{R}^{d}}\frac{1}{\delta^{d-1}}\int_{T_{\omega}^{\delta}(x)}|f(y)|\,dy,
	$$
	where $T_{\omega}^{\delta}(x)$ denotes the $\delta$-neighborhood of a unit line segment passing through $x$ with direction $\omega\in \mathbb{S}^{d-1}$.
	
	Similarly, the \emph{Nikodym maximal function} is defined by
	\begin{equation}\label{eq:nikodym-max}
		N_{\delta}f(x)=\sup_{\omega\in \mathbb{S}^{d-1}}\frac{1}{\delta^{d-1}}\int_{T_{\omega}^{\delta}(x)}|f(y)|\,dy.
	\end{equation}
	We now state the corresponding maximal function conjectures.
	
	\begin{conjecture}[Maximal Kakeya conjecture]
		For any $\epsilon>0$, there exists a constant $C_\epsilon>0$ such that
		\begin{equation}\label{K}
			\|K_{\delta}f\|_{q} \le C_{\epsilon} \delta^{1-\frac{d}{p}-\epsilon}\|f\|_{p},
		\end{equation}
		where $q=(d-1)p'$ and $1\leq p\leq d$. 
	\end{conjecture}
	
	\begin{conjecture}[Maximal Nikodym conjecture]
		For any $\epsilon>0$, there exists a constant $C_\epsilon>0$ such that
		\begin{equation}\label{N}
			\|N_{\delta}f\|_{q} \le C_{\epsilon} \delta^{1-\frac{d}{p}-\epsilon}\|f\|_{p},
		\end{equation}
		where $q=(d-1)p'$ and $1\leq p\leq d$.
	\end{conjecture}
	We denote by $K(p_{0})$ the statement that \eqref{K} holds for all $1\leq p\leq p_{0}$, and $N(p_{0})$ the statement that \eqref{N} holds for all $1\leq p\leq p_{0}$.  The condition $K(p_{0})$ implies that any Kakeya set has Hausdorff (and Minkowski) dimension at least $p_{0}$, while $N(p_{0})$ implies the same for Nikodym sets.
	
	Sogge \cite{Sogge99} defined the Nikodym maximal function on a Riemannian manifold $(M^d,g)$ by
	\begin{equation}\label{eq:nikodym-max-manifold}
		\mathcal N_{\delta}f(x)=\sup_{\gamma_{x}} \frac{1}{\delta^{d-1}}\int_{T_{\gamma_{x}}^{\delta}}|f(y)|\,dy,
	\end{equation}
	where $\gamma_{x}$ is a unit geodesic segment passing through $x$ and $T_{\gamma_{x}}^{\delta}$ denotes its $\delta$-neighborhood. We denote by $\mathcal N(p_{0})$ the statement: for any $\epsilon>0$, there exists a constant $C_\epsilon>0$ such that
	\begin{equation}\label{N_M}
		\|\mathcal N_{\delta}f\|_{q} \le C_{\epsilon} \delta^{1-\frac{d}{p}-\epsilon}\|f\|_{p},\quad q=(d-1)p',
	\end{equation}
	holds for all $1\leq p\leq p_{0}$.
	
	With Theorem \ref{theogeodesic} in hand, we now establish the following proposition.
	
	\begin{prop}\label{Prop}
		Let $(M^d,g)$ be a Riemannian manifold with constant sectional curvature. Then the following statements are equivalent:
		\begin{enumerate}
			\item The Kakeya maximal function estimate $K(p_{0})$ holds in $\mathbb{R}^d$.
			\item The Nikodym maximal function estimate $N(p_{0})$ holds in $\mathbb{R}^d$.
			\item The Nikodym maximal function estimate $\mathcal N(p_{0})$ holds on $(M^d,g)$.
		\end{enumerate}
		In particular, the Kakeya maximal conjecture in $\mathbb{R}^d$ is equivalent to the Nikodym maximal conjecture on manifolds with constant sectional curvature.
	\end{prop}
	
	\begin{proof}
		Every geodesic on a manifold with constant sectional curvature can be mapped to a straight line in $\mathbb{R}^d$ via a local diffeomorphism that is bi-Lipschitz (with uniformly bounded constants on compact sets). Consequently, the Nikodym maximal function on $(M^d,g)$, defined in \eqref{eq:nikodym-max-manifold}, is equivalent to the Nikodym maximal function in $\mathbb{R}^d$ (up to constant factors), and thus $\mathcal N(p_{0})$ is equivalent to $N(p_{0})$. 
		
		Moreover, as shown in \cite[Theorem 4.10]{Tao99}, the Kakeya and Nikodym maximal function estimates in $\mathbb{R}^d$, $K(p_{0})$ and $N(p_{0})$, are equivalent (in fact, the estimates imply one another up to a negligible loss in the exponents, which can be absorbed into the $\epsilon$ loss). Therefore, the three statements in the proposition are equivalent.
	\end{proof}

	\section{Nikodym-type sets on manifolds}\label{Nikodym type}
	\subsection{$(d,k)$-Nikodym sets on manifolds}
	
	We first recall the definition of $(d,k)$-Nikodym sets in $\mathbb{R}^{d}$.
	
	\begin{definition}
		A Borel set $N\subset\mathbb{R}^{d}$ with Lebesgue measure zero is called a \emph{$(d,k)$-Nikodym set} if for every 
		$
		x\in \mathbb{R}^{d}\setminus N,
		$
		there exists a $k$-plane $V$ passing through $x$ such that 
		$
		V\setminus\{x\}\subset N.
		$
	\end{definition}
	
	Analogously, we can define $(d,k)$-Nikodym sets on general Riemannian manifolds. To do so, we first introduce the notion of a locally geodesic submanifold via the exponential map. Without loss of generality, assume that the injectivity radius of the manifold is at least 10.
	
	\begin{definition}[$k$-locally geodesic submanifold]
		Let $M^d$ be a Riemannian manifold. For any point $p\in M$ and any set of $k$ linearly independent vectors in $T_p M$, these vectors locally span a $k$-dimensional submanifold via the exponential map. This submanifold is called a \emph{$k$-locally geodesic submanifold} at $p$. Moreover, its second fundamental form vanishes at $p$.
	\end{definition}
	
	\begin{definition}[$(d,k)$-Nikodym sets on manifolds]
		Let $M^d$ be a Riemannian manifold, and let $\Omega\subset M^d$ be a set of positive measure contained in a unit geodesic ball. For a fixed integer $k < d$, a set $N\subset M^d$ of measure zero is called a \emph{$(d,k)$-Nikodym set} if for every 
		$
		x\in \Omega,
		$
		there exists a $k$-locally geodesic submanifold $V$ at $x$ such that 
		$
		V\setminus\{x\}\subset N.
		$
	\end{definition}
	
	We now focus on $(d,k)$-Nikodym sets on manifolds of constant sectional curvature with dimension $d\ge 3$. In such manifolds every locally geodesic submanifold is totally geodesic.
	
	\begin{theorem}
		For a $d$-dimensional manifold with constant sectional curvature, $(d,d-1)$-Nikodym sets exist. Moreover, every $(d,d-1)$-Nikodym set has Hausdorff dimension $d$.
	\end{theorem}
	
	\begin{proof}
		Note that for manifolds with constant sectional curvature, each $k$-locally geodesic submanifold is totally geodesic. By the construction in Theorem \ref{theogeodesic}, any $(d-1)$-locally geodesic submanifold in a manifold with constant sectional curvature is mapped, via a local diffeomorphism, to a hyperplane in $\mathbb{R}^d$. Therefore, if $N\subset\mathbb{R}^d$ is a $(d,d-1)$-Nikodym set in the Euclidean setting (whose existence and full Hausdorff dimension were established by Falconer \cite{Falconer} and Mitsis \cite{Mitsis}), then its preimage under the straightening diffeomorphism is a $(d,d-1)$-Nikodym set in $M^d$. By Proposition \ref{propH}, this argument establishes both the existence and the fact that such a set must have full Hausdorff dimension $d$.
	\end{proof}
	
	\subsection{$(s,t)$-Furstenberg sets on Riemannian surfaces}
	
	We now generalize the $(s,t)$-Furstenberg problem to general Riemannian surfaces. Before defining $(s,t)$-Furstenberg sets on Riemannian surfaces, we first introduce the notion of a local model for the manifold of geodesics. As usual, we may assume that the injectivity radius of the surface is at least 10.
	
	\begin{definition}[A local model for the manifold of geodesics]\label{model}
		Let $M^2$ be a Riemannian surface, and let $\Omega$ be a local domain contained in a unit geodesic ball in $M^2$. Fix a geodesic segment $\gamma_0 \subset \Omega$ and choose a point $p \in \gamma_0$. By parallel transporting each vector $V \in S^1 \subset T_pM$ along $\gamma_0$, one obtains a unique corresponding direction along $\gamma_0$. In this way, the set $\gamma_0 \times S^1$ serves as a model for the local geodesics of $M^2$, representing (locally) all geodesics that intersect $\gamma_0$ in the directions determined by the parallel transport. The metric on this model is given by the product metric, where the first component is the restricted Riemannian metric on $\gamma_0$, and the second component is the standard metric on $S^1$.
	\end{definition}

	This is well-defined because parallel transport preserves the notion of “parallel” geodesics. In particular, if two geodesics $\gamma_1$ and $\gamma_2$ intersect $\gamma_0$ at points $p_1$ and $p_2$, respectively, and if their unit tangent vectors (after being parallel transported along $\gamma_0$) are parallel, then we say that $\gamma_1$ and $\gamma_2$ are parallel. Thus, the set $\gamma_0 \times S^1$ naturally parametrizes all local geodesics intersecting $\gamma_0$.
	
	If the Riemannian surface is the Euclidean plane, Definition \ref{model} corresponds to the family of affine $1$-subspaces of the plane (see \cite[Section 3.16]{Ma1999}).

	\begin{definition}[$(s,t)$-Furstenberg sets on Riemannian surfaces]
		For $s \in (0,1]$ and $t \in (0,2]$, an \emph{$(s,t)$-Furstenberg set} on a Riemannian surface is a set $E \subset \Omega$ (with $\Omega$ as above) satisfying the following property: there exists a family $\Gamma \subset \gamma_0 \times S^1$ of geodesics, with
		$$
		\dim_{\mathcal H}\Gamma \ge t,
		$$
		such that
		$$
		\dim_{\mathcal H}(E \cap \gamma) \ge s
		$$
		for every $\gamma \in \Gamma$. Here, the Hausdorff dimension $\dim_{\mathcal H}\Gamma$ is defined with respect to the natural metric on the model $\gamma_0 \times S^1$ for local geodesics of $M^2$.
	\end{definition}
	
	\begin{remark}
		$(s,t)$-Furstenberg sets on Riemannian surfaces generalize the notion of Nikodym sets on surfaces. In particular, a Nikodym set on a surface is always a $(1,1)$-Furstenberg set.
	\end{remark}
	
	Since the dimension of a Riemannian surface is 2, it is not affected by Kakeya compression phenomena. Therefore, it is natural to conjecture the following.
	
	\begin{conjecture}
		A $(s,t)$-Furstenberg set $E$ on a Riemannian surface satisfies
		$$
		\dim_{\mathcal H} E \ge \min\Bigl\{ s+t,\; \frac{3s+t}{2},\; s+1 \Bigr\}.
		$$
	\end{conjecture}

	As yet another consequence of Theorem \ref{theogeodesic}, we can easily establish the above conjecture for surfaces with constant Gaussian curvature using the recent breakthrough of Ren and Wang \cite{Ren-Wang} in Euclidean space.
	\begin{proof}[Proof of Theorem \ref{theoFur}]
		If the surface has constant curvature, then locally it behaves like a portion of $\mathbb{R}^{2}$, $\mathbb{S}^{2}$, or $\mathbb{H}^{2}$. By Theorem \ref{theogeodesic}, all geodesics in a small neighborhood can be mapped via a diffeomorphism to straight lines in the plane. In particular, by Proposition \ref{propH}, we map the $(s,t)$-Furstenberg set on a surface with constant Gaussian curvature to a $(s,t)$-Furstenberg set in the plane with identical Hausdorff dimension.
		
		More precisely, the model for local geodesics on the surface is given by $\gamma_0 \times S^1$, and there exists a diffeomorphism 
		$$
		Y: \gamma_{0} \times S^{1} \rightarrow l_{0} \times S^{1}
		$$
		which maps the geodesics on the surface to the corresponding elements in affine $1$-subspaces $l_{0} \times S^{1}$ in $\mathbb{R}^{2}$. Here,
		$$
		Y(z, e) = (\rho(z), \eta_z(e)),
		$$
		where $z$ is a parameter along the fixed geodesic $\gamma_0$, $\rho$ is the straightening local diffeomorphism restricted to $\gamma_0$, and $\eta_z$ is a  diffeomorphism from $S^1$ to $S^{1}$ that smoothly depends on $z$. One may easily check that $Y$ is a diffeomorphism, and thus bi-Lipschitz on any precompact subset.
		
		Since a bi-Lipschitz mapping preserves Hausdorff dimension (see Proposition \ref{propH}), the Hausdorff dimension of the family $\Gamma$ of geodesics is preserved under $Y$. Thus, the classical Furstenberg estimate in the Euclidean plane, as established by Ren and Wang \cite{Ren-Wang},
		$$
		\dim_{\mathcal{H}} E \ge \min\Bigl\{ s+t,\; \frac{3s+t}{2},\; s+1 \Bigr\},
		$$
		carries over directly to Riemannian surfaces with constant Gaussian curvature.
	\end{proof}

	\section{Translation-invariant phases under Bourgain's condition}\label{sec trans}
	\subsection{Proof of Theorem \ref{theomain}}
	We shall prove Theorem \ref{theomain} by proving a couple of lemmas.

	\begin{lemma}
		Let $\phi_{1},\, \phi_2 \in C^\infty(\mathbb{R}^{d-1})$ be two smooth functions, and suppose that $\phi_{2}$ has a non-degenerate Hessian. If there exists a smooth function $c(y)$ such that 
		\begin{equation}\label{mult}
			\nabla_{y}^{2} \phi_1(y) = c(y) \nabla_{y}^{2} \phi_2(y),
		\end{equation}
		then $c(y)$ is constant; that is, there exists a constant $c \in \mathbb{R}$ such that $c(y) \equiv c$.
	\end{lemma}
	
	\begin{proof}
		We first assume that $\phi_1$ has a non-degenerate Hessian; otherwise, if $\nabla_y^2 \phi_1(y)$ were degenerate, the equality \eqref{mult} would force $c(y) \equiv 0$, which is trivial.
		
		For convenience, denote 
		$$
		\partial_{ij}\phi_k(y) = \frac{\partial^2 \phi_k}{\partial y_i\partial y_j}(y), \quad k=1,2.
		$$
		Since both $\phi_1$ and $\phi_2$ have non-degenerate Hessians, for any fixed indices $l, m \in \{1,\dots,d-1\}$, we can choose indices $i,j$ such that the $2\times 2$ submatrix
		\begin{equation}\label{matrix}
			\begin{pmatrix}
				\partial_{il}\phi_k(y) & \partial_{im}\phi_k(y) \\
				\partial_{jl}\phi_k(y) & \partial_{jm}\phi_k(y)
			\end{pmatrix}
		\end{equation}
		is non-degenerate for $k=1,2$.
		
		By \eqref{mult}, we have
		$$
		\partial_{il}\phi_1(y) = c(y)\, \partial_{il}\phi_2(y) \quad \text{and} \quad \partial_{im}\phi_1(y) = c(y)\, \partial_{im}\phi_2(y).
		$$
		Differentiate the first equality with respect to $y_m$ and the second with respect to $y_l$:
		$$
		\partial_{ilm}\phi_1(y) = \partial_m c(y)\, \partial_{il}\phi_2(y) + c(y)\, \partial_{ilm}\phi_2(y),
		$$
		$$
		\partial_{ilm}\phi_1(y) = \partial_l c(y)\, \partial_{im}\phi_2(y) + c(y)\, \partial_{ilm}\phi_2(y).
		$$
		Subtracting these two equations yields
		$$
		\partial_m c(y)\, \partial_{il}\phi_2(y) - \partial_l c(y)\, \partial_{im}\phi_2(y) = 0.
		$$
		Similarly, one obtains
		$$
		\partial_m c(y)\, \partial_{jl}\phi_2(y) - \partial_l c(y)\, \partial_{jm}\phi_2(y) = 0.
		$$
		Since the submatrix in \eqref{matrix} (for $\phi_2$) is non-degenerate, it follows that
		$$
		\partial_l c(y) = \partial_m c(y) = 0.
		$$
		Because $l$ and $m$ are arbitrary indices, we conclude that $\nabla_y c(y) = 0$, which implies that $c(y)$ is constant.
	\end{proof}
	
	By applying the previous lemma, we can eliminate the dependence on $y$ in the constant using Bourgain's condition.
	
	\begin{corollary}\label{corophase}
		Let $\psi \in C^{\infty}(\mathbb{R}\times\mathbb{R}^{d-1})$ be a smooth function satisfying
		\begin{equation}\label{eq:rank}
			\operatorname{rank}\nabla_{y}^{2}\partial_{t}\psi(t;y)=d-1,
		\end{equation}
		and suppose there exists a smooth function $c(t;y)$ such that
		\begin{equation}\label{eq:Bourgain}
			\partial_{t}^{2}\nabla_{y}^{2}\psi(t;y)=c(t;y)\,\partial_{t}\nabla_{y}^{2}\psi(t;y).
		\end{equation}
		Then $c(t;y)$ depends only on $t$, that is, $c(t;y)=c(t)$.
	\end{corollary}
	Recall that our phase function is given by 
	$$
	\phi(x,t;y)=\langle x,y\rangle+\psi(t;y).
	$$
	By Bourgain's condition and H\"ormander's non-degeneracy condition, we have
	\begin{equation}\label{eq:main}
		\nabla_y^2 \partial_t^2 \psi(t; y) = c(t; y)\, \nabla_y^2 \partial_t \psi(t; y).
	\end{equation}
	From Corollary \ref{corophase}, it follows that
	\begin{equation}\label{eq:zero}
		\nabla_y^2 \Bigl( \partial_t^2 \psi(t; y) - c(t) \partial_t \psi(t; y) \Bigr) = \mathbf{0}.
	\end{equation}
	Hence, there exists a smooth vector-valued function $\mathbf{A}(t) \in \mathbb{R}^{d-1}$ such that
	\begin{equation}\label{eq:A}
		\partial_t^2 \partial_y \psi(t; y) - c(t) \partial_t \partial_y \psi(t; y) = \mathbf{A}(t).
	\end{equation}
	
	Next, consider the vector-valued function $\mathbf{B}(t)=(b_{1}(t),b_{2}(t),\cdots,b_{d-1}(t)) \in \mathbb{R}^{d-1}$ that satisfies
	\begin{equation}\label{eq:B}
		\mathbf{B}''(t) - c(t)\mathbf{B}'(t) + \mathbf{A}(t) = 0.
	\end{equation}
	We define the first coordinate transformation by
	$$
	T: (x, t) \mapsto \bigl(x + \mathbf{B}(t),\, t\bigr).
	$$
	Under this transformation, the phase changes to the form
	$$
	\phi\bigl(T(x, t); y\bigr) = \langle x,y\rangle + \tilde{\psi}(t; y),
	$$
	where
	$$
	\tilde{\psi}(t; y) = \mathbf{B}(t) \cdot y + \psi(t; y).
	$$
	We denote $\phi\bigl(T(x, t); y\bigr)$ by $\tilde{\phi}(x, t; y)$.
	
	Combining \eqref{eq:A} and \eqref{eq:B}, we obtain
	\begin{equation}\label{eq:final}
		\partial_{t}^{2} \partial_{y} \tilde{\psi}(t; y) = c(t) \partial_{t} \partial_{y} \tilde{\psi}(t; y).
	\end{equation}

	We conclude the proof of Theorem \ref{theomain} by illustrating that, through the following lemma, we can perform a second smooth change of variables in \( t \) to straighten the phase function.
	\begin{lemma}
		There exists a diffeomorphism $\alpha(t)$ such that
		$$
		\tilde \phi(x,\alpha(t);y)= \langle x,y\rangle + t\, h(y) + q(y)+ f(t),
		$$
		where $h(y)$ and $q(y)$ depend purely on $y$ and $f(t)$ depends purely on $t$.
	\end{lemma}
	
	\begin{proof}
		It suffices to show that there exists a diffeomorphism $\alpha(t)$ such that
		\begin{equation}\label{eq:straighten}
			\frac{d^2}{dt^2}\,\partial_{y}\tilde \psi\bigl(\alpha(t);y\bigr)=0.
		\end{equation}
		If \eqref{eq:straighten} holds, then integrating twice in $t$, we obtain
		$$
		\partial_{y}\tilde \psi\bigl(\alpha(t);y\bigr)= t\, H(y) + Q(y),
		$$
		where $H(y)$ and $Q(y)$ are conservative vector-valued functions in $\mathbb{R}^{d-1}$. A further integration in $y$ implies that
		$$
		\tilde \psi\bigl(\alpha(t);y\bigr)= t\, h(y) + q(y) + f(t),
		$$
		with $h(y)$ and $q(y)$ defined appropriately.
		
		To achieve \eqref{eq:straighten}, it suffices to find $\alpha(t)$ satisfying, via the chain rule,
		\begin{equation}\label{eq:chain}
			\alpha''(t)\,\partial_{t}\partial_{y}\tilde \psi\bigl(\alpha(t);y\bigr) + \partial_{t}^{2}\partial_{y}\tilde \psi\bigl(\alpha(t);y\bigr)\, (\alpha'(t))^{2} = 0.
		\end{equation}
		Recall that \eqref{eq:final} says
		\begin{equation}\label{eq:psiODE}
			\partial_{t}^{2}\partial_{y}\tilde \psi(t;y) = c(t)\,\partial_{t}\partial_{y}\tilde \psi(t;y).
		\end{equation}
		Substituting $t$ by $\alpha(t)$ in \eqref{eq:psiODE} and inserting into \eqref{eq:chain}, we deduce that \eqref{eq:chain} is equivalent to
		$$
		\alpha''(t) + c\bigl(\alpha(t)\bigr)\, (\alpha'(t))^{2} = 0,
		$$
		with initial conditions chosen so that $\alpha(0)=0$ and $\alpha'(0)\neq 0$.
		
		This is a nonlinear ordinary differential equation. Defining
		$$
		X_1(t)=\alpha(t),\quad X_2(t)=\alpha'(t),
		$$
		we can rewrite it as the first-order system
		\[
		\begin{cases}
			\displaystyle \frac{dX_1}{dt}=X_2, \\[1mm]
			\displaystyle \frac{dX_2}{dt}=-c\bigl(X_1\bigr)(X_2)^2.
		\end{cases}
		\]
		Since $c$ is smooth, the system is locally Lipschitz and the Picard–Lindelöf theorem guarantees the local existence (and uniqueness) of a solution. Thus, there exists a local diffeomorphism $\alpha(t)$ satisfying the desired property.
	\end{proof}
	
	Denote 
	$$\kappa(x,t):=T(x,\alpha(t)).$$ 
	It is easy to verify that $\kappa$ is a non-degenerate coordinate transformation. In fact, its Jacobian matrix is
	\[
	\begin{pmatrix}
		1 & 0 & \cdots & 0 &  b'_{1}\,\alpha'(t) \\
		0 & 1 & \cdots & 0 & b'_{2}\,\alpha'(t) \\
		\vdots & \vdots & \ddots & \vdots & \vdots \\
		0 & 0 & \cdots & 1 & b'_{d-1}\,\alpha'(t) \\
		0 & 0 & \cdots & 0 & \alpha'(t)
	\end{pmatrix}.
	\]
	Since $\alpha'(t)\neq 0$, the determinant of this matrix is nonzero, which shows that $\kappa$ is indeed a diffeomorphism. This completes the proof of Theorem \ref{theomain}.

	\begin{remark}\label{rem4}
		It is well known that one can ignore terms that depend purely on $y$ or purely on $(x,t)$ in the phase function of an oscillatory integral. In the curved Kakeya problem, the term $q(y)+f(t)$ does not alter the shape of the curve in a significant way. Thus, it is equivalent to consider the phase in the following form:
		$$
		\phi(x,t;y)=\langle x,y\rangle+t\, h(y)
		$$
		when the phase satisfies both Bourgain's condition and the translation-invariant condition.
	\end{remark}

	\subsection{Applications}
	
	Let us now turn to the associated curved Kakeya problems. We consider the curved Kakeya problem associated with a translation-invariant phase function $\phi$ that satisfies Bourgain's condition.
	
	By Theorem \ref{theomain} and Remark \ref{rem4}, there exists a local diffeomorphism under which the phase can be converted into the form
	$$
	\phi(x,t;y)=\langle x,y\rangle + t\,h(y).
	$$
	The curves determined by this phase satisfy
	$$
	x=\omega - t\,\partial_{y}h(y),
	$$
	where $\omega$ denotes the point through which the curve passes in the original coordinate plane. These curves are, in fact, straight lines. In particular, there exists a local diffeomorphism mapping the curved Kakeya sets to standard Kakeya sets, and this diffeomorphism preserves the Hausdorff dimension.
	
	Thus, all the known results for standard Kakeya sets—and consequently for Nikodym sets on manifolds with constant sectional curvature (see Table \ref{table1})—carry over to the corresponding curved Kakeya sets in this setting.

	As in Proposition \ref{Prop}, a similar conclusion can be drawn for the maximal function version of the problem. Let us recall the definition of the curved Kakeya maximal function \cite{Wisewell05}.
	
	\begin{definition}[Curved Kakeya maximal function]
		Given a phase function $\phi(\mathbf{x};y)$ that satisfies $(H_1)$ and $(H_2)$, and for $y \in \mathbb{B}_{\varepsilon_0}^{d-1}$ and $0<\delta<\varepsilon_0$, we define
		$$
		\mathcal{K}_\delta f(y) = \sup_{\omega \in \mathbb{B}_{\varepsilon_0}^{d-1}} \frac{1}{\delta^{d-1}} \int_{T_{y}^{\delta,\phi}((\omega,0))} |f|.
		$$
		Here $T_{y}^{\delta,\phi}(\omega)$ is as defined in Definition \ref{curved}.
	\end{definition}
	
	As in the standard Kakeya problem, $L^p\to L^q$ bounds for this maximal function yield lower bounds for the Hausdorff (and Minkowski) dimension of curved Kakeya sets. In particular, suppose that for any $\epsilon>0$, there exists a constant $C_\epsilon>0$ such that
	\begin{equation}\label{K'}
		\|\mathcal{K}_\delta f\|_{q} \le C_\epsilon\, \delta^{1-\frac{d}{p}-\epsilon}\|f\|_{p},\quad q=(d-1)p'.
	\end{equation}
	We denote by $\mathcal{K}(p_0)$ the statement that \eqref{K'} holds for all $1\le p\le p_0$.
	
	\begin{prop}\label{Prop4}
		If the phase satisfies both the translation-invariant condition and Bourgain's condition, then $\mathcal{K}(p_0)$ holds if and only if $K(p_0)$ holds.
	\end{prop}

	Next, we consider H\"{o}rmander's oscillatory operators. We fix $T_N$ to be an oscillatory integral operator as defined in Section \ref{sec intro}, with a phase function $\phi$ satisfying H\"{o}rmander's non-degeneracy condition. For such an operator, H\"{o}rmander conjectured that for all $\epsilon > 0$, there exists a constant $C_{\epsilon}$ such that
	\begin{equation}\label{Hor}
		\|T_{N} f\|_{L^p(\mathbb{R}^d)} \le C_{\epsilon}\,N^{-\frac{d}{p}+\epsilon}\, \|f\|_{L^p(\mathbb{B}_1^{d-1}(0))},
	\end{equation}
	for $p > \frac{2d}{d-1}$. H\"{o}rmander \cite{hormander} proved this conjecture for $d = 2$. For the higher-dimensional case, Stein \cite{stein1} proved \eqref{Hor} for $p \ge 2\frac{d+1}{d-1}$ and $d \ge 3$. Later, Bourgain \cite{Bourgain91} disproved H\"{o}rmander's conjecture by constructing a counterexample introducing the Kakeya compression phenomenon. Furthermore, his example showed that Stein's result is sharp in odd dimensions. For even dimensions, up to the endpoint case, Bourgain and Guth \cite{BG11} established the sharp result. In summary, we have the following theorem.
	
	\begin{theorem}[\cite{hormander,stein1,BG11}]\label{theo1}
		Let $d \ge 2$. For all $\epsilon > 0$ and $N \ge 1$, there exists a constant $C_{\epsilon}$ such that
		\[
		\|T_{N} f\|_{L^p(\mathbb{R}^d)} \le C_{\epsilon}\,N^{-\frac{d}{p}+\epsilon}\, \|f\|_{L^p(\mathbb{B}_1^{d-1}(0))},
		\]
		holds whenever
		\[
		p \ge \begin{cases}
			2\,\frac{d+1}{d-1} & \text{for $d$ odd}, \\[1ex]
			2\,\frac{d+2}{d}   & \text{for $d$ even}.
		\end{cases}
		\]
	\end{theorem}

	Furthermore, Guo, Wang, and Zhang \cite{GWZ22} conjectured that \eqref{Hor} should hold for all $p > \frac{2d}{d-1}$ provided that the phase function satisfies Bourgain's condition. By Theorem \ref{theomain}, if the phase function is translation-invariant and satisfies Bourgain's condition, then we can apply the results in \cite{BMV} or \cite{Guo-Oh} to obtain the following:
	
	\begin{theorem}
		Suppose that $\phi$ is translation-invariant and satisfies Bourgain's condition. Then for all $\epsilon > 0$, there exists a constant $C_{\epsilon}$ such that
		$$
		\|T_{N} f\|_{L^p(\mathbb{R}^3)} \le C_{\epsilon}\,N^{-\frac{3}{p}+\epsilon}\, \|f\|_{L^p(\mathbb{B}_1^{2}(0))},
		$$
		for $p > 3.25$.
	\end{theorem}
	
	The proof of the above theorem is straightforward. For any translation-invariant phase $\phi$ satisfying Bourgain's condition, Theorem \ref{theomain} ensures that there exists a local diffeomorphism transforming the phase into the form
	$$
	\tilde\phi(x,t;y) = \langle x,y\rangle + t\, h(y).
	$$
	Indeed, if the Hessian of $h$ is positive definite, then by the result in \cite{WangWu2024} the desired estimate holds for an even larger range, namely for $p > \frac{22}{7}$. On the other hand, if the Hessian of $h$ is negative definite, then by the result in \cite{BMV} (or \cite{Guo-Oh}) the estimate holds for $p > 3.25$. In general, since the Hessian of $h$ may be indefinite, the final estimate holds in the smaller range.

	Finally, let us briefly discuss a condition imposed on phase functions, known as the ``straight condition" introduced by the first author, Li, and Wang \cite{Gao}. This condition is sufficient to guarantee that the curves determined by the phase are all straight lines.

	\begin{definition}[Straight condition \cite{Gao}]\label{straight}
		We say that the phase $\phi$ satisfies the straight condition if 
		$$
		\frac{G_{0}(x,t;y)}{|G_{0}(x,t;y)|}
		$$ 
		remains invariant as $(x,t)$ varies.
	\end{definition}
	
	We compare the straight condition and Bourgain's condition through specific examples. On the level of phase functions, it is apparent that the straight condition is strictly stronger than Bourgain's condition. We denote the normalization of $G_0$ by $\tilde G_0$. Indeed, if $\phi$ satisfies the straight condition, then the normalized vector $G_0(x,t;y)/|G_0(x,t;y)|$ is independent of $(x,t)$; in other words,
	$$
	\tilde G_{0}(x,t;y)=\tilde G_{0}(\mathbf{0};y)
	$$
	for all $(x,t)$. Consequently, one obtains
	$$
	\Bigl((G_{0}\cdot \nabla)^2 \, \partial_{y}^{2}\phi(x,t;y)\Bigr)\Big|_{y=y_{0}}
	=\Bigl(\partial_{y}^{2}\bigl( G_{0}(\mathbf{0};y_0)\cdot \nabla )^{2}\phi(x,t;y)\bigr)\Bigr)\Big|_{y=y_{0}}=0,
	$$
	where we use the fact that
	$$
	G_{0}(\mathbf{0};y_{0})\cdot \partial_{y_{i}}\nabla\phi(x,t;y_{0})=0
	$$
	for any $1 \le i \le d$. Therefore, if a phase function satisfies the straight condition, it necessarily satisfies Bourgain's condition. 
	Theorem \ref{theomain} demonstrates that, in the translation-invariant setting, modulo a diffeomorphism, Bourgain's condition is equivalent to the straight condition. On the other hand, since Bourgain's condition is preserved under any diffeomorphism, phase functions that do not satisfy Bourgain's condition cannot be straightened.
	
	It is noteworthy that although the Euclidean distance function satisfies Bourgain's condition, it is not translation-invariant. Nevertheless, after an appropriate smooth change of variables, it does satisfy the straight condition (see \cite{Gao}). In other words, while the translation-invariant condition is sufficient to ensure that a phase function satisfying Bourgain's condition can be straightened, it is not necessary. 
	
	\begin{remark}
		To put Theorems \ref{theomain2} into context, note that it shows that for the Riemannian distance function on a manifold with constant sectional curvature, the associated curved Kakeya set (i.e. Nikodym sets) can be straightened via a diffeomorphism. However, this does not automatically imply that the resulting phase function after the smooth change of variables satisfies the straight condition in the sense of Definition \ref{straight}. This observation, together with our main results naturally raises the question: Which phase functions that satisfy Bourgain's condition can, modulo a smooth change of variables, yield Kakeya sets that are entirely straight? Moreover, for such phase functions, can we ensure that the transformed phases satisfy the straight condition? Theorems \ref{theomain2} and \ref{theomain} address these questions for two broad classes of phase functions and strongly suggest that most, if not all, curved Kakeya sets corresponding to a phase function satisfying Bourgain's condition can be straightened via an appropriate diffeomorphism.
	\end{remark}

	\bibliography{reference}
	\bibliographystyle{alpha}
\end{document}